\documentclass[table,a4paper]{article}

\usepackage[a4paper]{geometry}
\usepackage[english]{babel}
\usepackage[utf8]{inputenc}
\usepackage{amsmath}
\usepackage{amssymb}
\usepackage{amsfonts}
\usepackage{amsmath}
\usepackage{amsthm}
\usepackage{tikz}
\usepackage{pgfplots}
\usepackage{booktabs}
\usepackage{graphicx}
\usepackage{pgfplotstable}
\usepackage{tabularx}
\usepackage{xcolor}
\usepackage{enumerate}

\usepackage{url}
\usepackage{mathdots}
\usepackage{cite}

\newcommand{\dq}[1]{\mathrm{DQ}(#1)}

\providecommand*{\donothing}[1]{}

\newcommand{\kryl}[2]{{\mathcal K\mathcal M}_{#2}(#1)}

\DeclareGraphicsExtensions{.pdf,.eps,.png,.jpg}
\usepackage{epstopdf}

\newcommand{\wt}{\widetilde}
\newcommand{\wh}{\widehat}
\newcommand{\norm}[1]{\lVert #1 \rVert}

\definecolor{Gray}{gray}{0.95}

\newtheorem{thm}{Theorem}[section]
\newtheorem{prop}[thm]{Proposition}
\newtheorem{lem}[thm]{Lemma}
\newtheorem{cor}[thm]{Corollary}

\theoremstyle{definition}
\newtheorem{defn}[thm]{Definition}

\theoremstyle{remark}

\title{On the decay of the off-diagonal singular values in cyclic reduction\footnote{%
This work has been partially supported by an INdAM/GNCS Research Project 2016.
}}
\author{Dario A. Bini\footnote{%
    Dipartimento di Matematica, Unversit\`a di Pisa, L.go B. Pontecorvo, 5
    56127 Pisa, Italy, (\texttt{bini@dm.unipi.it})}, \ 
  Stefano Massei\footnote{%
    Scuola Normale Superiore di Pisa, P.za Cavalieri, 7, 56126, Pisa, Italy, 
    (\texttt{stefano.massei@sns.it})}, \
  Leonardo Robol\footnote{%
    KU Leuven, Dept. Computerwetenschappen (Postbus 02402), 
    Celestijnenlaan 200A, 3001 Heverlee (Leuven), Belgie,
    (\texttt{leonardo.robol@cs.kuleuven.be})}
}
\date{}

\renewcommand{\leq}{\leqslant}
\renewcommand{\geq}{\geqslant}

\pgfplotstableset{
	every head row/.style={before row=\toprule,after row=\midrule},
	clear infinite
}
	
\begin{document}
  \maketitle

  \begin{abstract}
  It was recently observed in \cite{netna} that the singular values of
  the off-diagonal blocks of the matrix sequences generated by the
  Cyclic Reduction algorithm decay exponentially. This property was
  used to solve, with a higher efficiency, certain quadratic matrix
  equations encountered in the analysis of queuing models. In this
  paper, we provide a sharp theoretical bound to the basis of this
  exponential decay together with a tool for its estimation based on a
  rational interpolation problem.  Applications to solving $n\times n$
  block tridiagonal block Toeplitz systems with $n\times n$
  semiseparable blocks and certain generalized Sylvester equations in
  $O(n^2\log n)$ arithmetic operations are shown.

  \bigskip 


  \bigskip 

    
\end{abstract}

  \section{Introduction}
Cyclic reduction, CR for short, is an algorithm originally introduced
by G.~H.~Golub and R.~W.~Hockney in \cite{hockney1965fast} for the
solution of certain block tridiagonal linear systems coming from the
finite difference discretization of elliptic PDEs. It has been later
generalized and extended to other contexts, like for instance to the
solution of polynomial matrix equations, and has been proven to be a
successful method for solving a large class of queuing problems and
infinite Markov Chains. We refer the reader to the books \cite{SMC},
\cite{bim:book} and to the survey paper \cite{CR} for more details and
for the many references to the literature.

Given three $m\times m$ matrices $A_{-1}$, $A_0$, $A_1$, and a
positive integer $n$ consider the block tridiagonal block Toeplitz
matrix $\mathcal A_n=\hbox{trid}_n(A_{-1},A_0,A_1)$ having block-size
$n$ where $A_0$ is on the main diagonal while $A_{-1}$ is in the lower
diagonal and $A_1$ in the upper diagonal. For a vector $b\in\mathbb
R^{mn}$, consider the system $\mathcal A_nx=b$.  Roughly speaking, CR
generates three sequences of $m\times m$ matrices $A_{-1}^{(h)}$,
$A_0^{(h)}$ and $A_1^{(h)}$, for $h=0,1,\ldots$, with $A^{(0)}_i=A_i$,
$i=-1,0,1$, and a sequence of systems $\mathcal A_{n_h}
x^{(h)}=b^{(h)}$, $\mathcal
A_{n_h}=\hbox{trid}_{n_h}(A_{-1}^{(h)},A_0^{(h)},A_1^{(h)})$, where
$n_h=\lfloor n_{h-1}/2\rfloor$ and $x^{(h)}$ is a subvector of
$x$. This way, solving a block tridiagonal block-Toeplitz system of
block size $n$ is reduced to solving a block tridiagonal block
Toeplitz system of size $\lfloor n/2\rfloor$.  The computation of
$A_i^{(h)}$ given $A_i^{(h-1)}$, for $i=-1,0,1$, amounts to perform
one matrix inversion and few matrix multiplications for the overall
cost per step of $O(m^3)$ arithmetic operations (ops).

Under certain assumptions, customarily verified in many applications,
the sequence $A_1^{(h)}$ and/or $A_{-1}^{(h)}$ converge doubly
exponentially to zero. This makes CR a powerful tool for solving
large or even infinite systems, as well as matrix equations of the
kind $A_{-1}+A_0X+A_1X^2=0$, typically encountered in the analysis of
queuing problems \cite{LatRam}, where the unknown is the $m\times m$
matrix $X$ and a solution of spectral radius at most $1$ is sought.

In short, the three sequences $A_i^{(h)}$, $i=-1,0,1$, which are
related to the Schur complements of certain principal submatrices of
$\mathcal A_n$, are given by the following matrix recurrences where we
report also two additional auxiliary sequences, namely $\widetilde
A^{(h)}$ and $\widehat A^{(h)}$, which have a role in the solution of
quadratic matrix equations and of linear systems when $n$ is not of
the kind $2^k-1$:
\begin{equation}\label{cr}
\begin{split}
&A_0^{(h+1)}=A_0^{(h)}-A_1^{(h)}S^{(h)}A_{-1}^{(h)}-A_{-1}^{(h)}S^{(h)}
A_1^{(h)},\quad S^{(h)}=(A_0^{(h)})^{-1}\\
&A_1^{(h+1)}=-A_1^{(h)}S^{(h)}A_1^{(h)},\quad A_{-1}^{(h+1)}=-A_{-1}^{(h)}S^{(h)}A_{-1}^{(h)},\qquad h=0,1,\ldots\\
&\widehat A^{(h+1)}=\widehat A^{(h)}-A_1^{(h)}S^{(h)}A_{-1}^{(h)},\quad \widetilde A^{(h+1)}=\widetilde A^{(h)}-A_{-1}S^{(h)}A_1^{(h)}
\end{split}
\end{equation}
with $A_0^{(0)}=\widetilde A^{(0)}=\widehat A^{(0)}=A_0$, $A_1^{(0)}=A_1$, $A_{-1}^{(0)}=A_{-1}$.

Here we assume that all the matrices $A_0^{(h)}$ generated by the
recursion are invertible so that CR can be carried out with no
breakdown. This assumption is generally satisfied in the applications.

In many cases of great interest, encountered for instance in the
analysis of bi-dimensional random walks, queuing models, network
analysis \cite{LatRam}, \cite{neuts}, \cite{Miya}, \cite{Miya1},
\cite{tandem}, and finite differences discretization of elliptic PDEs
\cite{chandra}, the blocks $A_{-1}$, $A_0$ and $A_1$ are tridiagonal
or, more generally, banded matrices.  This has raised great attention
to the computational analysis of this case.  The additional
tridiagonal structure makes it much cheaper to perform the first steps
of CR where the computational cost drops from $O(n^3)$ to $O(n)$
ops. However, after a few steps, the sparse structure of the initial
blocks is lost and one has to deal with full, apparently unstructured
matrices $A_i^{(h)}$, $i=-1,0,1$.

Recently, in \cite{netna}, it has been observed that if $A_{-1},A_0$
and $A_1$ are tridiagonal, then the matrices $A_i^{(h)}$, even if
dense, numerically maintain a property of {\em
  quasi-separability}. That is, their submatrices contained in the strict
upper triangular part or in the strict lower triangular part, called {\em
  off-diagonal submatrices}, have a ``small'' numerical rank. More
formally, it has been proved that if $\sigma_{i,h}$ are the singular
values of any off-diagonal submatrix of, say, $A_0^{(h)}$, ordered in
non-increasing order, then $\sigma_{i,h}\le \gamma t^{i/2}$ for some
small $\gamma>0$ and for some $0<t<1$. The value of $t$ is such that
the matrix $z^2A_1+zA_0+A_{-1}$ is invertible for any complex $z$ such
that $t<|z|<t^{-1}$.

The analysis of \cite{netna} provides a theoretical explanation of an
observed computational property which enables one to implement CR with
a high computational efficiency by relying on the properties of
quasiseparable matrices \cite{vanbarel:book1},
\cite{vanbarel:book2}. In fact, an efficient implementation of CR has
been given based on the software library \cite{borm} of hierarchical
quasiseparable matrices, and the numerical experiments show the high
effectiveness of this approach. 

However, the results of \cite{netna} provide an under estimate of the
decay properties of the singular values of the off-diagonal blocks of
$A_{-1}^{(h)}$, $A_0^{(h)}$ and $A_1^{(h)}$. In fact, it turns out
that, even in the cases where the matrix polynomial
$z^2A_1+zA_0+A_{-1}$ is singular at some point just outside a thin
annulus $\mathbb A_t=\{z\in\mathbb C:\quad t\le |z|\le t^{-1}\}$
obtained with some $t$ very close to 1, the observed exponential decay
of the singular values is still evident with a basis of the
exponential much smaller than the given theoretical bound $t$.

A typical example is given by the discrete Laplacian matrix where
$A_0=\hbox{trid}(-1,4,-1)$, $A_{-1}=A_1=-I$ so that
$t=1-1/(n+1)+O(1/(n+1)^2)$. In this case, for moderately large values
of $n$, the bound $t^j$ is still close to 1 for values of $j$ as large
as $n$.  As a result, the plot of the upper bounds to the singular
values would be an almost horizontal line.  On the other hand from the
numerical experiments it turns out that the decay of the singular
values is still exponential despite the width of the annulus collapses
to zero, and the basis of the exponential is much less than $t$ and
almost independent of $n$.

With the tools introduced in this paper, we can capture this property
as shown in Figure \ref{fig:poisson} where the decay of the
off-diagonal singular values of the matrix $H_0=\lim_h A_0^{(h)}$,
together with their theoretical upper bounds, are displayed.

   \begin{figure} 
    	\begin{center}
    	\begin{tikzpicture}
    	\begin{semilogyaxis}[
    	  xlabel=$l$,
    	  width=.72\linewidth,
    	  height=.36\textheight,
    	  ylabel=Singular values ($\sigma_l$),
    	  ymin=1e-20,
    	  xmax=25,
    	  legend pos = south west]
    	\addplot table[x index=0, y index = 1] {poisson.dat};
    	\addplot table[x index=0, y index = 2] {poisson.dat};
    	\addplot[domain=1:25,green,mark=triangle] 
    	  {260.65 * (0.995)^(x-1) };
    	\legend{Singular values, Proposition~\ref{prop:poisson},
    		Bound from~\cite{netna}}
    	\end{semilogyaxis}
    	\end{tikzpicture}
    	\end{center}   	
        \caption{This graph displays the singular values of the
          largest off-diagonal block of the matrix $H_0=\lim_h
          A_0^{(h)}$ computed by means of CR for the Poisson problem
          where $m=200$. The red squares denote the upper bound, the
          blue discs the computed values, the green triangles the
          bound from \cite{netna}. The exponential decay and the
          sharpness of the bound are evident.}
    	\label{fig:poisson}
     \end{figure}
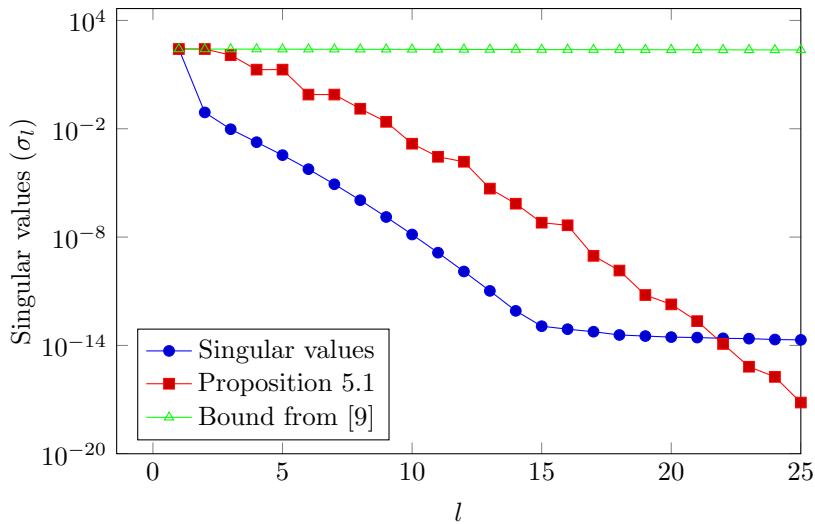
In fact, in this paper we provide a different theoretical explanation
of the exponential decay of the singular values which relies on an
unpublished result of B.~Beckermann \cite{beck2004} where the decay of
certain singular values associated with some Krylov matrix is
expressed in terms of the accuracy of a rational function
approximation problem.

This analysis leads to a fast algorithm, which we call {\em
  quasiseparable CR} (QCR for short), for solving an $n\times n$ block tridiagonal
block Toeplitz system, where the matrix $\hbox{trid}_n(B,A,C)$ has
quasiseparable, say tridiagonal, $m\times m$ blocks $A,B,C$. The cost
of the algorithm is $O(mn\log m+m\log^2 m\log n)$ ops which reduces to
$O(n^2\log n)$ ops for $m=n$. This cost is comparable with that of the
fast Poisson solvers \cite[Sect. 4.8.4]{golub-vanloan}, which apply to
the case where $A=\hbox{trid}_n(-1,4,1)$, $B=C=-I$. Unlike the latter
algorithms, quasiseparable CR covers a wider and more general class of
cases.

We show also an application of QCR to solving a generalized Sylvester
equation of the kind
\[
\sum_{i=1}^k A_iX B_i=C
\]
for given matrices $A_i,B_i$ and $C$ of compatible sizes, in the case
where $A_i$ are tridiagonal Toeplitz, and $B_i$ are quasiseparable
matrices. In fact, in this case, the problem is reduced to solve a
block tridiagonal block Toeplitz system with quasiseparable
blocks. The cost of the solution is again $O(n^2\log n)$ ops where,
for simplicity, we assume that all the matrices involved are $n\times
n$.

Decay properties of the off-diagonal blocks of matrix functions have
been recently received much attention. In particular, in the paper by
M.~Benzi and P.~Boito and N.~Razouk \cite{benzi2013decay} the decay
properties of spectral projectors associated with large and sparse
Hermitian matrices are investigated. In \cite{kresuns} D. Kressner and A. Susnjara prove a priori bounds for the numerical rank of the off-diagonal blocks of spectral projectors ---associated with symmetric banded matrices--- by using the best rational approximant of the sign function.  In \cite{benzi2014decay}
M.~Benzi and P.~Boito extend previous results on the exponential
off-diagonal decay of the entries of analytic functions of banded and
sparse matrices to the case where the matrix entries are elements of a
$C^*$-algebra.  M.~Benzi and V.~Simoncini \cite{benzi-simoncini} find
decay bounds for completely monotonic functions of matrices which are
the Kronecker sum of banded or sparse matrices.  While C.~Canuto,
V.~Simoncini and M.~Verani \cite{canuto} analyze the decay pattern of
the inverses of banded matrices of the form $I\otimes M+M\otimes I$
where $M$ is tridiagonal, symmetric and positive definite.  In
\cite{chandra}, S.~Chandrasekaran, P.~Dewilde, M.~Gu, and
N.~Somasunderam analyze the numerical rank of the off-diagonal blocks
in the Schur complements of block tridiagonal block Toeplitz systems
discretizing bi-dimensional elliptic equations.

The paper is organized as follows. In Section \ref{sec:prel} we
provide some preliminary results including the main properties of CR,
its functional interpretation, and the definitions of
$k$-quasiseparable matrices and of hierarchical $k$-quasiseparable
matrices.
Section \ref{sec:laurent} concerns the analysis of the properties of
the matrix coefficients in the Laurent expansion of the matrix
function $\psi(z)=\varphi(z)^{-1}$, where
$\varphi(z)=z^{-1}A_{-1}+A_0+zA_1$. In fact, this matrix function
captures the structural and computational properties of CR. Its domain
of analyticity is the annulus $\mathbb A_t$ whose width has been used
in \cite{netna} to prove the exponential decay. The main result of
this section is Lemma \ref{lem:addends} where we show that any
off-diagonal block of $\psi(z)$ can be written as the sum of 4 terms;
each term is the product of a Krylov matrix and of a transposed Krylov
matrix.

In Section \ref{sec:disp} ---relying on a result by B.~Beckermann--- we
provide a bound to the singular values of a matrix which satisfies a
suitable displacement equation. Then we apply this result to find
sharp bounds to the singular values of the off-diagonal blocks of
$\psi^{(h)}(z)$ and we extend these bounds to the block $A_i^{(h)}$
and to the limit value $\lim \psi^{(h)}(z)=A_0^{(\infty)}$. The main results of this section are given in Theorems \ref{thm:thm} and \ref{prop:limit}.

Section \ref{sec:valid} deals with the experimental validation of the
theoretical bounds to the decay.
In Section \ref{sec:appl} we show applications of the quasiseparable
CR to solving block tridiagonal block Toeplitz systems and to solving
certain generalized Sylvester equations.  We report also the results
of some numerical experiments where the above applications are
tested. Finally, Section \ref{sec:conclusion} draws the conclusions.

   \section{Some preliminaries}\label{sec:prel}
Throughout, $\mathbb Z$ and $\mathbb N$ denote the set of relative
integers and of natural numbers, respectively, while $\mathbb R$ and
$\mathbb C$ denote the complex and the real field, respectively.  We
recall the fundamental properties of CR and of quasiseparable
matrices.
   
Cyclic reduction can be formulated in functional form in terms of two
matrix Laurent series $\varphi(z)$ and $\psi(z)$, namely,
   \[
     \varphi(z) = z^{-1} A_{-1} + A_0 + z A_1, \qquad 
     \psi(z) = \varphi(z)^{-1},
   \]
where $\psi(z)$ is defined in the set where $\varphi(z)$ is
invertible.  Here and hereafter, we assume that $\varphi(z)$ is
invertible in the annulus $\mathbb A_t=\{z\in\mathbb C:\quad t\le
|z|\le t^{-1}\}$ for some $0<t<1$. This assumption is generally
verified in the applications. In certain cases, by means of scaling
the matrices $A_{1}$ and $A_{-1}$ by suitable constants $\alpha$ and
$\alpha^{-1}$, respectively, one can meet this assumption. Throughout
we denote $\mathbb T=\mathbb A_1$ the unit circle in the complex
plane.
   
We recall the following property which is fundamental for our
analysis, see for instance \cite{SMC} and \cite{netna}.

\begin{prop}\label{prop:1}  
Let $A_{-1},A_0,A_1$ be $n\times n$ matrices such that CR can be
carried out. Define
$\varphi_h(z)=z^{-1}A_{-1}^{(h)}+A_0^{(h)}+zA_1^{(h)}$, where
$A_i^{(h)}$ are the matrices generated by \eqref{cr}, and set
$\psi_h(z)=\varphi_h(z)^{-1}$. Then
   \[
     \psi^{(h)}(z^{2^h}) = \frac{1}{2^h} \sum_{j = 1}^{2^h}
     \psi(\xi_{2^h}^j z)
   \]
where $\xi_{2^h}$ is a primitive $2^h$-th root of the unity.
   \end{prop}
   
The following definitions are fundamental to formalize the fast decay
of the singular values of the off-diagonal submatrices generated by
CR. We say that an $m\times m$ matrix $A$ is $k$-{\em quasiseparable}
if all the submatrices contained in the strict upper triangular part or in
the strict lower triangular part have rank at most $k$ and there exists at
least one submatrix with rank $k$. We say also that $k$ is the {\em
  quasiseparable rank} of $A$.
   
We say that $A$ is \emph{hierarchically $k$-quasiseparable} if either
$m\leq k$ or there exists a $2 \times 2$ block partitioning of the
matrix such that the diagonal blocks are square and have size $\lfloor
\frac m2\rfloor$ and $\lceil \frac m2 \rceil$, respectively, the
off-diagonal blocks have rank at most $k$ and the diagonal blocks are
hierarchically $k$-quasiseparable.  Moreover, in this recursive
partitioning there exists an off-diagonal submatrix of rank exactly $k$.

 This partitioning leads
to the simplest hierarchical representation, known in the literature
as \emph{hierarchically off-diagonal low rank} (HODLR), which is the
one exploited in \cite{netna} for speeding up CR.
         
The following result states that if the singular values of the
off-diagonal blocks of $A$ decay {\em fast}, then $A$ is close to a
hierarchical quasiseparable matrix. That is, for a relatively small $k$  there is a perturbation $\delta A$
of {\em small} norm such that $A+\delta A$
is hierarchically $k$-quasiseparable.

       \begin{thm}\label{thm:perturb}
         Let $f(l)$ be a function over the positive integers, and let
          $A\in\mathbb C^{m\times m} $ be a matrix such that
           $\sigma_l(B)\leq f(l)$
         for every off-diagonal block $B$ in $A$. 
         Then, for any $l$  
         there exists a perturbation matrix 
         $\delta A$ such that 
         $A + \delta A$ is hierarchical quasiseparable of rank at most $l$
   	     and 
    	  $\norm{\delta A}_2 \leq  f(l) \cdot  \log_2 m$.
       \end{thm}
       
       \begin{proof} 
First, recall that if the nonzero singular values of a matrix $B$ are
$\sigma_1\ge\sigma_2\ge\ldots\ge \sigma_k$ then, for any $j\le k$ we
may write $B$ as a matrix of rank $j$ plus a perturbation $\delta B$
such that $\|\delta B\|_2=\sigma_{j+1}$. Now consider an HODLR like partitioning of $A$ with minimal blocks of dimension $l$.  Notice that the depth of this
recursive partition is $\sigma=\lceil \log_2(\frac{m}{l})\rceil$.
This way, for each off-diagonal block $B$ of this partitioning and for
any integer $j$, there exists a perturbation matrix that makes this
block of rank $j$.  The 2-norm of this perturbation is equal to
$\sigma_{j+1}(B)\le f(j+1)$.  We may form the matrix $\delta A$ which
collects all these perturbations of each off-diagonal block of the
above partitioning. This way, if $j=l-1$, the off-diagonal blocks of
$A + \delta A$ have rank at most $l$.  We can now show that
$\norm{\delta A}_2 \leq f(l) \cdot \log_2 m$. We have
         \[
           \delta A = \sum_{i = 0}^{\sigma} \delta A_i,\qquad \sigma\leq \log_2 m,
         \]
 where $\delta A_i$ is the correction obtained by putting together all
 the blocks at level $i$ of subdivision, that is,
         \[
           \delta A_0 = 
             \begin{bmatrix}
               0 & \delta X_1^{(0)} \\
               \delta X_2^{(0)} & 0 \\
             \end{bmatrix}, \qquad 
           \delta A_1 = \left[\begin{array}{cc|cc}
             0 & \delta X_1^{(1)}& & \\
             \delta X_2^{(1)} & 0& & \\\hline
             &&0 & \delta X_3^{(1)} \\
             && \delta X_4^{(1)} & 0\\
           \end{array}\right], \qquad \ldots \ 
         \]
         
 Since the summands are just permutations of block diagonal matrices
 their $2$-norm is the maximum of the $2$-norms of the (block)
 diagonal entries, and this gives the desired bound.
       \end{proof}

Thus, our aim is to prove that the matrix function $\psi^{(h)}(z)$
defined in Proposition \ref{prop:1}, has off-diagonal blocks with
singular values which decay {\em exponentially} to zero so that the
assumptions of Theorem \ref{thm:perturb} are satisfied with
$f(l)=e^{-\alpha l}$ for some positive $\alpha$. This decay property
is then extended to $\varphi_h(z)$ by inversion and finally to the
blocks $A_{-1}^{(h)}, A_0^{(h)},A_1^{(h)}$ by means of
interpolation. More details on this technique are given in
\cite{netna}.

The estimates of the parameter $\alpha$ given in the paper
\cite{netna} depend on the value $t$ which defines the domain $\mathbb
A_t$ of invertibility of the matrix $\varphi(z)$. If $t$ gets close to
$1$, then $\alpha$ takes values close to 0, and the theoretical bound
of the exponential decay loses its sharpness. Here, we introduce a
different analysis which better fits with the decay observed in the
numerical experiments.
   
We define the following class of problems for which the matrices
$A_i^{(h)}$, $i=-1,0,1$ generated by CR through \eqref{cr} have the
exponential decay of the singular values in their off-diagonal blocks
at any step $h$ of the iteration.
   
    \begin{defn}
Let $\varphi(z) = z^{-1}A_{-1} + A_0+ z A_1$, where $A_{-1},A_0,A_1$
are $m\times m$ matrices with entries in $\mathbb C$, be such that CR
can be applied with no breakdown by means of \eqref{cr}.  Let $f(l)$
be a positive function in $l^1(\mathbb N)$.  We say that $\varphi(z)$
is \emph{$f$-decaying-quasiseparable} if, $\forall h \in \mathbb N$,
$\forall z\in \mathbb T$ and for every off-diagonal block $\wt
C^{(h)}(z)$ of $\psi^{(h)}(z)$, we have
         \[
           \sigma_l(\wt C^{(h)}(z)) \leq  \norm{\psi^{(h)}(z)}_2 \cdot f(l) , 
         \]
where $\sigma_l(\wt C^{(h)}(z))$ denotes the $l$-th singular value of
the matrix $\wt C^{(h)}(z)$.  We define the set of such matrix
functions $\varphi(z)$ as $\dq{f}$.
       \end{defn}

%
 
\section{Laurent coefficients of an off-diagonal block}\label{sec:laurent}
In this section, we consider the matrix Laurent series expansion of
$\psi(z)$, that is, $\psi(z)=\sum_{i=-\infty}^{+\infty}z^iH_i$ for
$z\in\mathbb A_t$, which exists and is convergent since $\psi(z)$ is
analytic in the domain $\mathbb A_t$ where $\varphi(z)$ is analytic
and non-singular. Then we will analyze the properties of the
coefficients of an off-diagonal block of this Laurent series.

We define the eigenvalues of $\varphi(z)$ as the roots of the
polynomial $p(z)=\det (A_{-1}+zA_0+z^2A_1)$.  Observe that if $\det
A_1\ne 0$ the polynomial $p(z)$ has degree $d=2m$ so that there are
$2m$ roots. If, on the other hand, $\det A_1=0$ then $d<2m$ and for
this reason, we add to the $d$ roots of $p(z)$ other $2m-d$ roots at
the infinity. In this way we can say that $\varphi(z)$ has always $2m$
eigenvalues including possible eigenvalues at the infinity.
     
 Here we assume that the eigenvalues $\xi_i$, $i=1,\ldots,2m$ of
 $\varphi(z)$ satisfy the balanced splitting property with respect to
 the unit circle
              \begin{equation}\label{splitting}
                   |\xi_1|\leq\dots\leq |\xi_m|<t<1<t^{-1}<|\xi_{m+1}|\leq\dots\leq |\xi_{2m}|. 
                   \end{equation}
We call $t$ the radius of the splitting.  The splitting property
\eqref{splitting} is needed to guarantee the applicability of CR and
that the convergence to zero of the blocks $A_{-1}^{(h)}$ and
$A_1^{(h)}$ is doubly exponential \cite{SMC}.

Consider the following partitioning of $\psi(z)$ and $\varphi(z)$
       \[
       \varphi(z)=\left(\begin{array}{cc}A(z)&B(z)\\ C(z)&D(z)\end{array}\right), \qquad \psi(z)=\left(\begin{array}{cc}\wt A(z)&\wt B(z)\\ \wt C(z)&\wt D(z)\end{array}\right)=\left(\begin{array}{cc} S_{D}(z)^{-1} & *\\ -D(z)^{-1}C(z)S_{D}(z)^{-1}&*\end{array}\right),
       \] 
where the diagonal blocks are square,
$S_{D}(z)=A(z)-B(z)D(z)^{-1}C(z)$ is the Schur complement of $D(z)$,
and $*$ denotes blocks which are not relevant for our analysis.
       
Moreover, suppose that the splitting \eqref{splitting} holds also for
the eigenvalues of $D(z)$ --this is true for problems from stochastic
models which are ruled by M-matrices-- and assume that the matrix
coefficients $ A_i$ have quasiseparable rank $k$ for $i=-1,0,1$. These
hypotheses are always satisfied for a large class of important
problems like not null recurrent Quasi Birth-Death problems (QBDs)
with banded blocks, up to rescaling the coefficients \cite{SMC}. This
guarantees that the matrix functions $\varphi(z)$ and $D(z)$ are
invertible in the annulus $\mathbb A_{t}$ for some $t<1$.

Observe that, since the off-diagonal blocks of $A_i$ have rank at most
$k$ for $i=-1,0,1$, then any off-diagonal block $C(z)$ of $\varphi(z)$
can be written as 
\[ C(z) = z^{-1} U_{-1} V_{-1}^t + U_0 V_0^t + z U_1
V_1^t, \qquad \norm{U_i} = \norm{A_i}, \quad \norm{V_i} = 1,
         \]
where $U_i$ and $V_i$ have $k$ columns and the superscript $t$ denotes 
transposition. 

Defining 
      \[
      U=\begin{bmatrix}
      U_{-1}&\vline& U_0&\vline&U_1
      \end{bmatrix},
      \quad 
      V(z)=\begin{bmatrix}
       z^{-1}V_{-1}&\vline& V_0&\vline&zV_1
       \end{bmatrix},
      \] 
we can write $\wt C(z)=-\wt U(z) \wt V(z)^t$, where $\wt
U(z)=D(z)^{-1} U$ and $\wt V(z) =S_{D}(z)^{-t}V(z)$.  Observe that
$S_{D}(z)^{-1}$ is the upper left diagonal block of $\psi(z)$. This
gives us a crucial information on the coefficients of the matrix
Laurent series expansion of $D(z)^{-1}$ and $S_{D}(z)^{-1}$. In order
to perform this analysis we have to recall a general result which
provides an explicit expression of the coefficients $H_i$ of the
Laurent expansion of $\psi(z)$.
     
     \begin{thm}[Part of Theorem 3.20 in \cite{SMC}]
       \label{thm:fourier-coefficients-kylov}
Let $\varphi(z)=z^{-1}A_{-1}+A_0+zA_1$ with $A_i\in\mathbb{R}^{m\times
  m}$, $i=-1,0,1$ and assume that the eigenvalues $\xi_i$,
$i=1,\dots,2m$ of $\varphi(z)$ satisfy \eqref{splitting}.  Moreover
suppose that there exist $R$ and $\wh R$ with spectral radius less
than $1$ which solve the matrix equations
     \begin{align}\label{eq:mateqR}
     A_1+XA_0+X^2A_{-1}&=0,\\
     X^2A_1+XA_0 +A_{-1}&=0,
     \end{align}
respectively.  Then there exist $G$ and $\wh G$ solutions of the
reversed matrix equations
     \begin{align}\label{eq:mateqG}
       A_1X^2+A_0X +A_{-1}&=0,\\
       A_1+A_0X+A_{-1}X^2&=0,
     \end{align}
respectively, with spectral radius less than 1. Moreover, expanding
$\varphi(z)^{-1}=\sum\limits_{j=-\infty}^{+\infty}z^j H_j$ yields
     \[
     H_j=
     \begin{cases}
     H_0 \wh R^{-j} = G^{-j} H_0  & j\leq0\\
     H_0  R^j = \wh G^j H_0 & j\geq0
     \end{cases}, \quad 
     G = H_{-1} H_0^{-1},~ \wh G = H_1 H_0^{-1},~ 
     R = H_0^{-1} H_1,~ \wh R = H_0^{-1} H_{-1}. 
     \]
The spectrum of $G$ and $\wh R$ is formed by the eigenvalues of
$\varphi(z)$ inside the unit disc, the spectrum of $\wh G$ and $R$ is
formed by the reciprocals of the eigenvalues of $\varphi(z)$ outside
the unit disc.
     \end{thm}
     
This result, applied with $\varphi(z)=D(z)$ and combined with what
said previously, tells us that the Laurent coefficients of $\wt U(z)$
are of the form
     \[
      D^{-1}(z) = \sum_{j \in \mathbb Z} z^j H_{D,j}, 
                 \qquad 
                 H_{D,j} =
                 \begin{cases}
                   G_D^{-j} H_{D,0}   & j\leq0,\\
                   \wh G_D^j H_{D,0}  & j\geq0,
                 \end{cases} 
     \]
where $G_D$ and $\wh G_D$ are the solutions of the matrix equations
associated with $D(z)$ of the kind \eqref{eq:mateqG} and
     \[
            S_D(z)^{-1} = \sum_{j \in \mathbb Z} z^j H_{S,j}, 
            \qquad 
            H_{S,j} =
            \begin{cases}
              [ I \ 0 ] H_{ 0} \wh R^{-j} [ I \ 0 ]^t  & j\leq0,\\
              [ I \ 0 ] H_{ 0}  R^j [ I \ 0 ]^t & j\geq0,
            \end{cases}
          \]
where the latter equation is obtained by applying
Theorem~\ref{thm:fourier-coefficients-kylov} to the original matrix
Laurent polynomial $\varphi(z)$.

Consider the simpler case where $k=1$ and the decomposition of each
off-diagonal block $C(z)$ of $\varphi(z)$ can be written as $C(z) = u
v^t$ (a constant dyad). This is not restrictive since, in the other
cases, we can write $C(z)$ as a linear combination of at most $3k$
terms of the above form with coefficients $z^j$, $j=-1,0,1$.
       
In view of Theorem~\ref{thm:fourier-coefficients-kylov}, for $z \in
\mathbb T$ we can write each off-diagonal block $\wt C(z)$ of
$\psi(z)$ as
       \[
         \wt C(z) = \wt u(z) \wt v(z)^t, \qquad  
         \wt u(z) = \sum_{j \geq 0} \wh G_D^j H_{D,0} u z^j
         + \sum_{j < 0} G_D^{-j} H_{D,0} u z^j,
       \]
where $\wt v(z) = S_D(z)^{-t}v$, the matrix function $S_{D}(z)^{-1}$
is the inverse of the Schur complement of $D(z)$ and $\norm{v}_2 \leq
1$.  Observe that the Laurent coefficients of $\wt u(z)$ corresponding
to positive powers of $z$ lie in the Krylov subspace $\mathcal K_j(\wh
G_D, H_{D,0} u)$, while the coefficients corresponding to the negative
powers are in $\mathcal K_j(G_D, H_{D,0} u)$.  Here we denote by
$\mathcal K_j(A,v)$ the $(j+1)$-dimensional Krylov subspace
          \[
          \mathcal K_j(A,v)=\hbox{span}(v,Av,A^2v,\ldots,A^{j}v).
          \]
          
Analogously we know that
        \[
        v^tS_D(z)^{-1}=\left(\sum_{j\geq 0}\wh v^tH_{\psi,0}  R^j z^j+\sum_{j< 0}\wh v^t H_{\psi,0} \wh R^{-j} z^j\right)
        \begin{bmatrix}I\\ 0 \end{bmatrix},
          \qquad \wh v:= \begin{bmatrix} v \\ 0 \end{bmatrix},
        \]
therefore
        \begin{equation}\label{sum}
\begin{split}
   -\wt C(z)=\left(
        \sum_{j \geq 0}\right.&\left. 
        \wh G_D^j H_{D,0} u z^j
            + \sum_{j < 0} G_D^{-j} H_{D,0} u z^j
\right)
\\
 & \cdot\left(\sum_{j> 0}\wh v^tH_{\psi,0}  R^j z^j+
 \sum_{j\leq  0}\wh v^t H_{\psi,0}  \wh R^{-j} z^j\right)
            \begin{bmatrix}
              I \\ 0 
            \end{bmatrix}.
\end{split}
        \end{equation}
Denoting by $\wt C^{(h)}(z^{2^h})$ the corresponding off-diagonal
sub-block in $\psi^{(h)}$, from Proposition \ref{prop:1} we have
     \begin{equation}\label{ctilde}
     \wt C^{(h)}(z^{2^h})=\frac{1}{2^h}
       \sum_{j=1}^{2^h}\wt C(z\zeta_{2^h}^j).
     \end{equation}

 In the following, the matrices with columns 
 of the form $A^j b$, for some matrix $A$ and a
 vector $b$, which we  call
 \emph{Krylov matrices}, will play an important role. 
 We indicate a Krylov matrix with the notation
 \[
   \kryl{A,b}{n} := \left[\ b\ \vline \ Ab\ \vline \ \dots \   
    \vline A^{n-1}b\  \right].  
 \]
Moreover, we  denote by $J$ the counter-identity matrix of appropriate
size such that $[1,2,\ldots,n]J=[n,n-1,\ldots,1]$.
    
Relying on \eqref{sum} we can prove the following result.
     
     \begin{lem}\label{lem:addends}
If $C(z)=uv^t$, then $-\wt C^{(h)}(z^{2^h})$ is the sum of the
following four outer products:
     \begin{equation} \begin{split}\label{addends}
     -\wt C^{(h)}(z^{2^h}) &= \Big[ 
       \kryl{\wh G_D, \wh a}{2^h} \cdot \kryl{ \wh R^t, \wh b}{2^h}^t \\
       &+ 
        z^{2^h-1}\cdot \kryl{ \wh G_D, \wh a}{2^h} \cdot J \cdot \kryl{ R^t, b}{2^h}^t \\
       &+ z^{1-2^h}\cdot\kryl{ G_D, a}{2^h} \cdot J \cdot \kryl{ \wh R^t, \wh b}{2^h}^t \\
       &+ 
       \kryl{G_D, a}{2^h} \cdot \kryl{R^t, b}{2^h}^t
       \Big] \begin{bmatrix} I \\ 0 \end{bmatrix},
\end{split}      \end{equation}
          where 
          \begin{align*}
          a=& \left(\sum_{s\in 2^h\mathbb Z\cap \mathbb N} z^{-s-1} G_D^{s+1}\right) H_{D,0} u,&
           b=&\left(\sum_{s\in 2^h\mathbb Z\cap \mathbb N} z^{s+1}R^{s+1} \right)^tH_{\psi,0}^t\wh v,\\
           \wh a=&\left(\sum_{s\in 2^h\mathbb Z\cap \mathbb N} z^s \wh G_D^s\right) H_{D,0} u,& \wh b=&\left(\sum_{s\in 2^h\mathbb Z\cap \mathbb N} z^{-s}\wh R^s \right)^tH_{\psi,0}^t\wh v.
          \end{align*}
     \end{lem}
     \begin{proof}
Thanks to \eqref{sum} we may write $-\wt C(z)$ as the sum of four
outer products. By the linearity of \eqref{ctilde} we can consider
them separately.  Take for example
        \[
        \left(\sum_{j \geq 0} \wh G_D^j H_{D,0} u z^j\right)\cdot\left(\sum_{j\leq 0}\wh v^tH_{\psi,0}  \wh R^{-j} z^j\right)=\sum_{j \geq 0} \wh G_D^j H_{D,0} u \wh v^t H_{\psi,0}\sum_{s\geq 0}\wh R^s z^{j-s},
        \]
where we have ignored $[ I \ 0 ]^t$ because it can be factored on the
right.  The block $\wt C^{(h)}(z)$ of $\psi^{(h)}(z)$ corresponding to
$\wt C(z)$ in $\psi(z)$ verifies the relation $\wt C^{(h)}(z^{2^h}) =
\frac{1}{2^h} \sum_{l = 1}^{2^h} \wt C(z \zeta_{2^h}^l)$ so that
       \begin{align*}
         \frac{1}{2^h}&\sum_{l=1}^{2^h}\sum_{j \geq 0} \wh G_D^j H_{D,0} u \wh v^t H_{\psi,0}\sum_{s\geq 0}\wh R^s (z\zeta_{2^h}^l)^{j-s}=\sum_{j \geq 0} \wh G_D^j H_{D,0} u \wh v^t H_{\psi,0}\sum_{s\in (2^h\mathbb Z+j)\cap \mathbb N}\wh R^s z^{j-s}\\
         &=\sum_{j = 0}^{2^h-1}\left(\sum_{s\in (2^h\mathbb Z+j)\cap \mathbb N} z^s \wh G_D^s\right) H_{D,0} u \wh v^t H_{\psi,0}\left(\sum_{s\in (2^h\mathbb Z+j)\cap \mathbb N}\wh R^s z^{-s}\right),
       \end{align*}
where $2^h \mathbb Z+j := \{ s \in \mathbb Z \ | \ s \equiv j \mod 2^h
\}$.  Observe that the $(j+1)$-st term of the previous sum is equal to
the $j$-th term multiplied on the left by $z \wh G_d$ and on the right
by $z^{-1}\wh R$.  In particular we can rewrite it as
       \[\left[
                 a\ \vline\ z \wh G_D\cdot a\ \vline\ \dots\ \vline\ (z \wh G_D)^{2^h-1}\cdot a
                 \right]\cdot
                 \left[
                     \wh b\ \vline\ (z^{-1} \wh R^t)\cdot \wh b\ \vline\ \dots\ \vline\ (z^{-1} \wh R^t)^{2^h-1}\cdot\wh b
                     \right]^t,
                     \]
                     that is,
       $
       \mathcal{KM}_{2^h}(z\wh G_D,a)\cdot\mathcal{KM}_{2^h}(z^{-1}\wh R^t,\wh b)$.

The variables $z$ in the above factors cancel out, and we obtain one
of the addends in the statement of the theorem.
       
Then consider $\left(\sum_{j \geq 0} \wh G_D^j H_{D,0} u
z^j\right)\cdot\left(\sum_{j> 0}\wh v^tH_{\psi,0} R^{-j} z^j\right)$
for which we arrive at the expression
      \[
      \sum_{j = 0}^{2^h-1}\left(\sum_{s\in (2^h\mathbb Z+j)\cap \mathbb N} z^s \wh G_D^s\right) H_{D,0} u \wh v^t H_{\psi,0}\left(\sum_{s\in (2^h\mathbb Z-j)\cap \mathbb N} R^s z^{s}\right).
      \]
This time we have a product of the form
          \[
         \left[
          a\ \vline \ z\wh G_D\cdot a\ \vline\dots\vline (z \wh G_D)^{2^h-1}\cdot a\
          \right]\cdot
          \left[
              (z R^t)^{2^h-1}\cdot b \ \vline\dots\vline (z R^t)\cdot b\ \vline \ b\
              \right]^t,
          \]
that is
       $
       z^{2^h-1}\cdot \mathcal{KM}_{2^h}(\wh G_D,a)\cdot J\cdot \mathcal{KM}_{2^h}(\wh R^t,\wh b)$.
The other two relations are obtained in a similar manner.
     \end{proof}

In the case $C(z)=z^suv^t$ with $s=-1,1$ one can recover the same
behavior just taking into account a shift in the powers of $z$ in
\eqref{sum} that modifies the powers of $z$ in the outer products
accordingly.

     
  \section{Singular values and displacement rank}\label{sec:disp}
Lemma \ref{lem:addends} provides a tool for analyzing the singular
values decay of the off-diagonal block $\wt C^{(h)}(z)$ of
$\psi^{(h)}(z)$.  In fact, these blocks can be written as the sum of
few terms each of them is the product of two Krylov matrices, one of
which is transposed.

The next step is to investigate the singular values of a product of
this kind. In this analysis, we rely on the concept of displacement
rank and on some result by B.~Beckermann \cite{beck2004}, of which we
report the proof.

     \begin{defn}
Given matrices $A,B,X\in\mathbb C^{m\times m}$ the \emph{displacement
  rank} of $X$ with respect to the pair $(A,B)$ is defined as
     \[
     \rho_{A,B}(X)=\hbox{rank}(AX-XB).
     \]
     \end{defn}
We need also to introduce the set $\mathcal R_{n,d}$ of rational
functions over $\mathbb C$ where $n$ and $d$ are the degree of the
numerator and of the denominator, respectively.

For a matrix $X$ with a small displacement rank it is possible to provide bounds on its
singular values in terms of the optimal values of some Zolotarev problems \cite{zol} according to the following result of
B.~Beckermann \cite{beck2004}. 
     
     \begin{thm}\label{thm:beckermann}
Let $X\in\mathbb C^{m\times m}$ and suppose that there exist two
normal matrices $A,B\in\mathbb C^{m\times m}$ such that $
\rho_{A,B}(X)=d.$ Then, indicating with $E$ and $F$ the spectrum of
$A$ and $B$ respectively, for the singular values $\sigma_i(X)$ of $X$
it holds:
     \[
     \frac{\sigma_{1+l\cdot d}(X)}{\norm{X}_2}\leq Z_l(E,F):=\inf_{r(x)\in \mathcal R_{l,l}}\frac{\max_{x\in E} |r(x)|}{\min_{x\in F} |r(x)|},\quad l=1,2,\dots\ .
     \]
     \end{thm}
     \begin{proof}
Consider $p(x):=\sum_{i=0}^lp_ix^i$ and $q(x):=\sum_{i=0}^lq_ix^i$
polynomials of degree $l$ and define $r(x):=\frac{p(x)}{q(x)}$. We
prove that the matrix
     \[
     \Delta:=q(A)Xp(B) - p(A)Xq(B)
     \]
has rank at most $l\cdot d$. Without loss of generality we consider
the case $d=1$ and suppose $AX-XB=u v^t$. We can prove by induction
that $A^kX-XB^k=\sum_{h=0}^{k-1}A^huv^tB^{k-1-h}$. For $k=1$ the
property trivially holds. For $k>1$ one has:
     \begin{align*}
     A^kX=A^{k-1}XB+A^{k-1}uv^t&=XB^k+\left(\sum_{h=0}^{k-2}A^h
     uv^tB^{k-1-h}\right)B+A^{k-1}uv^t\\&=XB^k+\sum_{h=0}^{k-1}A^huv^tB^{k-1-h}.
     \end{align*}
     Now, observe that 
     \[
     \Delta=q(A)Xp(B) - p(A)Xq(B)=\sum_{i\neq j}^d(q_ip_j-q_jp_i)(A^iXB^j-A^jXB^i),
     \]
     and if $i>j$ (the other case is analogous)
     \[
     A^iXB^j-A^jXB^i=A^j(A^{i-j}X-XB^{i-j})B^j=A^j\left(\sum_{h=0}^{i-j-1}A^huv^tB^{k-1-h}\right)B^j.
     \]
In particular all the addends involved in the expansion of $\Delta$
can be expressed as sum of dyads whose left vectors belong to the
Krylov space $\mathcal K_l(A,u)$ and so it has rank at most $l$.
     
     Assume that $q(A)$ and $p(B)$ are invertible, define $Y:=q(A)^{-1}\Delta p(B)^{-1}$ observe that 
     $
     X-Y=r(A)Xr(B)^{-1}$ so that
     \[
     \norm{X-Y}_2=\norm{r(A)Xr(B)^{-1}}_2\leq \norm{X}_2 \max_{E}|r(x)|\max_F|r(x)|^{-1}=\norm{X}_2 \frac{\max_{E}|r(x)|}{\min_F|r(x)|}.
     \]
Since $\sigma_{k+1}(X)$ coincides with the minimum of $\|X-W\|_2$
taken over all the matrices $W$ of rank $k$, and since
$\hbox{rank}(Y)=\hbox{rank}(\Delta)\le l\cdot d$, we find that
\[
\sigma_{l\cdot d+1}(X)\le \|X-Y\|_2\le \norm{X}_2 \frac{\max_{E}|r(x)|}{\min_F|r(x)|}.
\]    
Taking the infimum over the set of rational functions of degree
$(d,d)$ completes the proof.
     \end{proof}
     
Note that the normality hypothesis can be relaxed by replacing it with
the diagonalizability of $A$ and $B$. The price to pay is a larger
constant depending on the conditioning of the eigenvector matrices as stated by the following

     \begin{cor}\label{cor:beckermann}
Let $X\in\mathbb C^{m\times m}$ and suppose that there exist two
diagonalizable matrices $A,B\in\mathbb C^{m\times m}$ such that $
\Delta_{A,B}(X)=d$, that is $A=V_AD_AV_A^{-1}$, $B=V_BD_BV_B^{-1}$
with $D_A$ and $D_B$ diagonal matrices.  Then, indicating with $E$ and
$F$ the spectrum of $A$ and $B$ respectively, it holds:
          \[
          \sigma_{1+l\cdot d}(X)\leq Z_l(E,F)\cdot \norm{X}_2\cdot
          \kappa(V_A)\cdot \kappa(V_B)\] 
where $\kappa(W)=\|W\|_2\|W^{-1}\|_2$ denotes the spectral condition
number of $W$.
     \end{cor}
     
     The case where $E$ and $F$ are disjoint subsets of the real line,  has been extensively studied by Zolotarev \cite{zol} who managed to provide explicit bounds for $Z_l(E,F)$. The result we are going to quote is adapted to our case and can be found in \cite{guttel2015zolotarev}. See also \cite{akh1990,abr1984,med2005} for more classical references.
                \begin{thm}[Zolotarev]\label{thm:zolotarev}
                    Let $\delta\in(0,1)$, $E:=[-\infty,-\delta^{-1}]\cup[\delta^{-1},+\infty]$ and $F=[-\delta,\delta]$. Then
                    \[
                    Z_{2l}(E,F)\leq \frac{2\rho^l}{1-2\rho^l},
                    \]
                    where 
                    \begin{align*}
                    \rho:=\operatorname{exp}\left(-\frac{\pi K(\sqrt{1-\delta^4})}{2K(\delta^2)}\right),\qquad K(x):=\int_0^1\frac{1}{\sqrt{(1-t^2)(1-x^2t^2)}}dt.
                    \end{align*}
                    Moreover, if $\delta \approx 1$ then $K(\delta^2)\approx  \log\left(\frac{4}{\sqrt{1-\delta^4}}\right)$ and $K(\sqrt{1-\delta^4})\approx \frac{\pi}{2}$, yielding
                    \[
                    Z_{2l}(E,F)\leq \frac{2\rho^l}{1-2\rho^l}\approx \frac{2\wt \rho^l}{1-2\wt \rho^l},\qquad \wt{\rho}:=\operatorname{exp}\left(-\frac{\pi^2}{2\log\left(\frac{16}{1-\delta^4}\right)}\right).
                    \]
                    \end{thm}
     
Now, we prove that some matrices involved in the decomposition of the 
off-diagonal submatrices of $\psi^{(h)}(z)$ enjoy a small displacement
rank.

       \begin{prop}\label{prop:displacement}
 Under the assumptions and the notation of Lemma~\ref{lem:addends} we
 have 
 \[
 \wt C^{(h)}(z^{2^h})=\left[I \ I \right]\cdot X^{(h)}(z)
 Y^{(h)}(z)^t\cdot\left[I\ I\right]^t\cdot \left[I\ 0\right]^t
 \]
  where
       \[
       X^{(h)}(z):=\begin{bmatrix}
       \kryl{ \wh G_D, \wh a}{2^h} \\
       z^{1-2^h} \kryl{G_D, a}{2^h} J \\
       \end{bmatrix},\qquad Y^{(h)}(z):=\begin{bmatrix}
              z^{2^h-1}\kryl{ R^t, b}{2^h} J \\
              \kryl{\wh R^t,\wh b}{2^h} 
              \end{bmatrix}.
       \]    
 Moreover, we have  the following
 displacement relations:
      \[
      \rho_{W_D,\Pi}(X^{(h)})=1,\qquad \rho_{W,\Pi}(Y^{(h)})=1,
      \]
      with
      \[
      W_D:=\begin{bmatrix}
      \wh G_D&0\\
      0&  G_D^\dagger
      \end{bmatrix},\qquad W:=\begin{bmatrix}
             (R^\dagger)^t&0\\
            0& \wh R^t
            \end{bmatrix},\qquad \Pi=\begin{bmatrix}
            0&&&1\\
            1&\ddots\\
            &\ddots&\ddots\\
            &&1&0
            \end{bmatrix},
      \]
      where the super-script $\dagger$ indicates the Moore-Penrose pseudoinverse.
       \end{prop}  
       \begin{proof}
 The first claim simply follows by expanding the expression for $\wt C^{(h)}(z^{2^h})$ and by
 comparing it with equation~\eqref{addends}.  Concerning the
 displacement equations, a direct computation shows that the matrices
       \[
       W_D X^{(h)}(z)-X^{(h)}(z)\Pi\quad\text{and}\quad W Y^{(h)}(z)-Y^{(h)}(z)\Pi\
       \]
       have only the last column possibly different from zero.
       \end{proof}
       
The above result allows us to give a bound to the singular values of $\wt C^{(h)}(z)$.
       
    \begin{thm}\label{thm:thm}
 Let $\varphi(z)=z^{-1}A_{-1}+A_0+zA_1$ be an $m\times m$ matrix
 Laurent polynomial such that the CR ---given by \eqref{cr}--- can be carried
 out with no breakdown, the splitting property \eqref{splitting} is
 verified, and $\varphi(z)$ has quasiseparable rank $1$ for every
 $z\in\mathbb T$. Assume that the matrices $R$ and $\wh R$ which solve
 the matrix equations \eqref{eq:mateqR} are diagonalizable by means of
 the two eigenvector matrices $V_R$ and $V_{\wh R}$,
 respectively. Assume that $G_D$ and $R$ are invertible. Then $\varphi(z) \in \dq{f}$ where
                  \[
                  f(l):=\gamma\cdot Z_l(E,\mathbb T),
                  \]
 with $\gamma$ a multiple of $\max\{\kappa(V_R),\kappa(V_{\wh R})\}$
 and $E$ contains the eigenvalues of $\varphi(z)$.
         \end{thm}
         \begin{proof}
Notice that a generic off-diagonal matrix $\wt C^{(h)}(z)$ in
$\psi^{(h)}(z)$ can be seen as a submatrix of
         \[
           [ I \ I ] X^{(h)}(z)Y^{(h)}(z)^t \begin{bmatrix}
             I \\ I 
           \end{bmatrix}. 
         \] 
In view of Proposition~\ref{prop:displacement} we know that
$Y^{(h)}(z)$ has displacement rank $1$. The displacement relation for
$Y^{(h)}(z)$ involves the matrices $W$ and $\Pi$ whose eigenvalues
correspond to those of $\varphi(z)$ and to the roots of the unity of
order $2^h$, respectively. Moreover, $W$ is diagonalizable by means of
$V_W:=\mathrm{diag}(V_R^{-t},V_{\wh R}^{-t})$. Therefore, applying
Corollary~\ref{cor:beckermann} we can write
         \[
         \sigma_{1+l}(Y^{(h)}(z))\leq  Z_l(E,\mathbb T)  \cdot \norm{Y^{(h)}(z)}_2 \cdot \kappa(V_W).
         \]
Since $W$ is block-diagonal we have
$\kappa(V_W)=\max\{\kappa(V_R),\kappa(V_{\wh R})\}$. In particular we
can bound the singular values of $\wt C^{(h)}(z)$ with the quantity
         \[
          \sigma_{1+l}(\wt C^{(h)}(z))\leq  2 \cdot Z_l(E,\mathbb T)\cdot\norm{X^{(h)}(z)}_2\cdot  \norm{Y^{(h)}(z)}_2 \cdot \kappa(V_W).
         \]
Defining $\gamma:=2\cdot \kappa(V_W) \cdot \max\limits_{h\in\mathbb N,z\in
  \mathbb T}\frac{2 \norm{X^{(h)}(z)}_2\cdot
  \norm{Y^{(h)}(z)}_2}{\norm{\psi^{(h)}(z)}_2}$ we get the
thesis.
         \end{proof}
         
The constant $\gamma$ in the previous theorem is an index of how much
the factorization $X^{(h)}(z)Y^{(h)}(z)^t$ is unbalanced. This
limitation is not present in the following result which describes
the asymptotic behavior as $h\to\infty$. It is possible to show that
the block diagonal terms in $W^{(h)}(z)=X^{(h)}(z)Y^{(h)}(z)^t$ quickly
decay to $0$ in practice, making the following bounds numerically
accurate after a few steps.

     \begin{thm}\label{prop:limit}
     Let $W^{(h)}(z)=X^{(h)}(z)Y^{(h)}(z)^{t}$, where $X^{(h)}(z)$ and $Y^{(h)}(z)$ are the matrices defined in Proposition \ref{prop:displacement}.
Then $\lim_{h\to\infty} W^{(h)}(z)=W^{(\infty)}$ has the following block partitioning
\[
W^{(\infty)}=\begin{bmatrix}0&B_1\\B_2&0\end{bmatrix}
\]
where the diagonal blocks are square and the off-diagonal blocks are independent of $z$. Moreover,
 we have
     $
     \rho_{V_D,V}(W^{(\infty)})=2,
     $
     where
     \[
     V_D:=\begin{bmatrix}
     \wh G_D&0\\ 0& G_D
     \end{bmatrix}\quad \text{and}\quad V:=\begin{bmatrix}
          R^\dagger&0\\ 0& \wh R^\dagger
          \end{bmatrix}.
     \]
If the matrices $G_D,\wh{G}_D,R$ and $\wh{R}$ are diagonalizable by means of $V_{G_D},V_{\wh{G}_D},V_{R}$ and $V_{\wh{R}}$, respectively, then, indicating with $\wt C$ the off-diagonal block  in $H_{\psi,0}$ corresponding to $\wt C(z)$ we have the following bounds to its singular values 
     \[
     \sigma_{1+2l}(\wt C)\leq \gamma\cdot Z_l(E,F), \qquad 
     \gamma := 2 \cdot \max \{ \kappa(V_G), \kappa(V_{\wh G}) \} \cdot
     \max \{ \kappa(V_R), \kappa(V_{\wh R}) \} \cdot 
     \norm{\wt C}_2, 
     \]
     where $E$ contains the eigenvalues of $\varphi(z)$ and $D(z)$ inside the unit disc while $F$ contains those outside.
     \end{thm}
     \begin{proof}
From the definition of $X^{(h)}$ and $Y^{(h)}$ one has
\[
W^{(h)}(z)=\begin{bmatrix}
z^{2^h-1}\mathcal{KM}_{2^h}(\wh G_D,\wh a)J(\mathcal{KM}_{2^h}(R^t,b))^t&
\mathcal{KM}_{2^h}(\wh G_D,\wh a)(\mathcal{KM}_{2^h}(\wh R^t,\wh b))^t\\
\mathcal{KM}_{2^h}(G_D, a)(\mathcal{KM}_{2^h}(R^t, b))^t&
z^{1-2^h}\mathcal{KM}_{2^h}(G_D,a)J(\mathcal{KM}_{2^h}(\wh R^t,\wh b))^t.
\end{bmatrix}
\]
Since the spectral radii of the matrices $R$, $\wh R$, $G_D$ and $\wh G_D$ are less than $1$, then the block diagonal entries of 
     $W^{(h)}$ tend to zero as $h\to\infty$ and the two off-diagonal blocks have limits $B_1$ and $B_2$, respectively. More precisely
     \[
     W^{(\infty)}=\begin{bmatrix}
     0&B_1\\
     B_2&0
     \end{bmatrix},\quad B_1=\sum_{i\ge 0}\wh G_D^i\wh a\wh b^t\wh R^i,\quad B_2=\sum_{i\ge 0} G_D^iab^t R^i.
     \] 
Thus, we have
     \[
     \wh G_D B_1 -B_1 \wh R^\dagger=\wh a\wh b^t \wh R^\dagger.
     \]
     An analogous argument holds for $B_2$, and gives
     the rank-$2$ displacement.      
     The matrix $\wt C$ can be written as $\wt C = \left[I \ I \right]\cdot W^{(\infty)} \cdot\left[I\ I\right]^t\cdot \left[I\ 0\right]$, which corresponds to an off-diagonal block of $\lim_{h\to\infty}\psi^{(h)}(z)$. Due to the recurrence relation  $\psi^{(h+1)}(z^2)=\frac{1}{2}(\psi^{(h)}(z)+\psi^{(h)}(-z))$, this limit is equal to the central coefficient $H_{\psi,0}$ in the series expansion of $\psi(z)$. 
     The thesis follows by applying Corollary~\ref{cor:beckermann}. 
     \end{proof}

   \section{Experimental validation of the results}\label{sec:valid}

This section is devoted to verify the previous results by means of
numerical experiments.  We do that by computing numerical estimates of
the bound given in Theorem~\ref{prop:limit} together with the
singular values of the off-diagonal blocks of $H_{\psi,0}$.  The
actual bounds are obtained by choosing a particular family of
rational functions that suit the considered problem. We will see that,
even if our choices are relatively simple, and not optimal, they
already provide sharp decay bounds in practice.

 	As a first example, we consider instances of the problem coming from the framework of Markov chains i.e., the sum $A_{-1}+A_0+I+A_1$ is sub-stochastic, that is, it has non-negative entries and the sum along each row is at most 1. In particular, the matrices $A_{-1},I+A_0$ and $A_1$ have non negative entries  and are scaled in order to satisfy the splitting assumption~\eqref{splitting} (see also Section 4.3 in \cite{netna}). 
 	

We select dense $300\times 300$-blocks generated at random and such that $\varphi(z)$
is of quasiseparable rank $1$. For satisfying the latter hypothesis we
impose that the strictly triangular parts of the blocks are the
restrictions of dyads with the same left vectors.
 	
 	\begin{figure}
 \includegraphics[width=.48\linewidth]{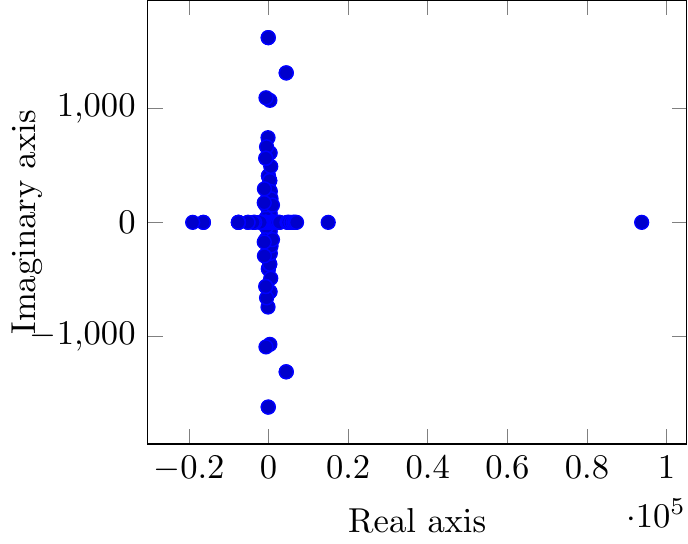}~\includegraphics[width=.48\linewidth]{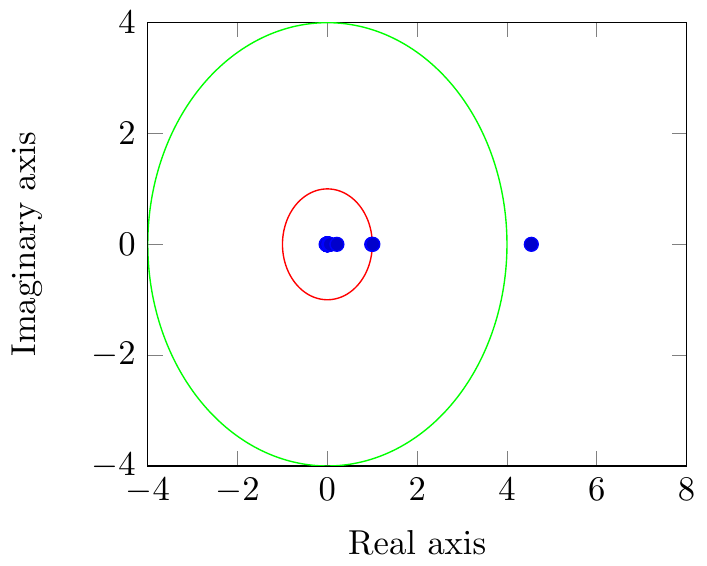}
   
 		\includegraphics[width=.48\linewidth]{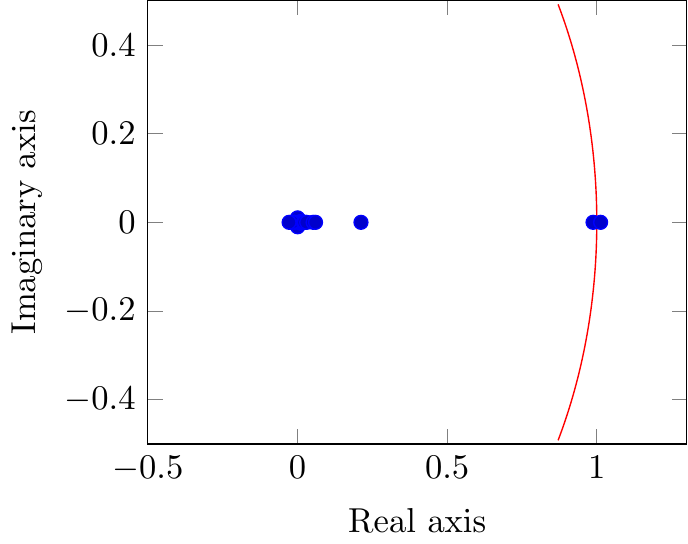}~\begin{tikzpicture}
 			\begin{semilogyaxis}[xlabel=$j$,ylabel=Singular values,width=.48\linewidth,
 			ymin=1e-14,xmax=25]
 			\addplot table[x index = 0, y index = 1] {quasi.dat};
 			\addplot table[x index = 0, y index = 2] {quasi.dat};
 			\legend{Singular values, Theorem~\ref{prop:limit}};
 			\end{semilogyaxis}
 	 			\end{tikzpicture}
 	 			\caption{Singular value decay in $H_0$ and the bound given
 	 				by Theorem~\ref{prop:limit}. }
 	 			\label{fig:quasidecay}
 		\end{figure}
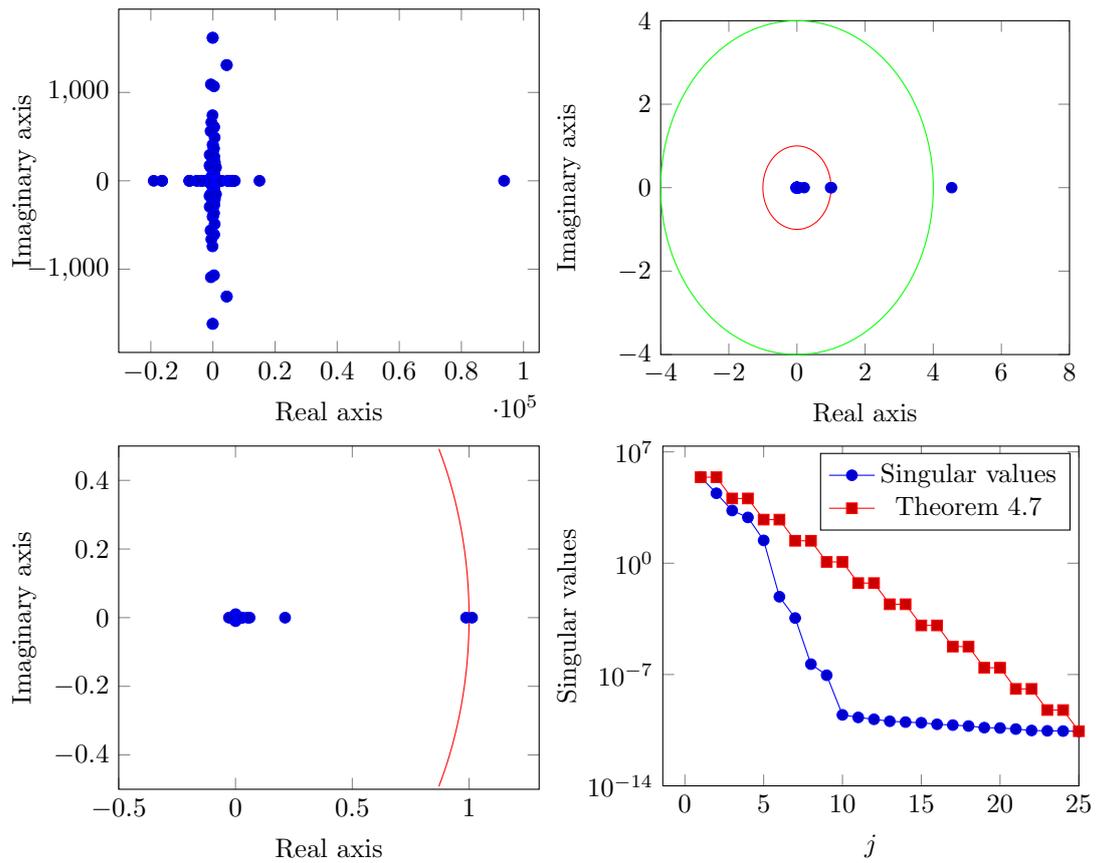
 		
We divide the resulting distribution of the eigenvalues in three
cluster. One is contained in a neighborhood of $0$, another is in the
complement of the disc of radius $4$ and finally we have two
eigenvalues close to $1$, $\lambda_1$ and $\lambda_2$, inside and
outside the unit circle, respectively.
 		
Motivated by this, we choose the sequence of rational function
 		\[
 		r_l(z):=\frac{z-\lambda_1}{z-\lambda_2}z^{l-1},
 		\]
for roughly estimating the Zolotarev problem. The results are shown in Figure~\ref{fig:quasidecay}.
 		
%
%


As a second example, we consider the linear system arising from the
discretization of a 2D Poisson equation, whose matrix is block
tridiagonal with the following form:
	\[
	  T = \begin{bmatrix}
	    \tilde A & \tilde C \\
	    B & A & C \\
	    & \ddots & \ddots & \ddots \\	 
	    & & B & A & C \\
	    & & & \hat B & \hat{A} \\   
	  \end{bmatrix}, \qquad 
	  A = \begin{bmatrix}
	    4 & -1 \\
	    -1 & 4 & -1 \\
	    & \ddots & \ddots & \ddots \\
	    & & -1 & 4 & -1 \\
	    & & & -1 & 4 \\
	  \end{bmatrix}, \qquad B = C = -I.
	\]
The above system can be solved by means of the cyclic reduction. The
eigenvalues of the associated $\varphi(z)$ can be computed explicitly
and one can easily check that they are real positive and provide a
splitting $t=1-1/(n+1)+O(1/(n+1)^2)$.  The matrices $A, B$ and $C$ are
very special instances of $1$-quasiseparable matrices, so we can state
a refined version of Theorem~\ref{prop:limit}, which gives a
smaller displacement rank for the limit case.
    
\begin{prop} \label{prop:poisson}
      Let $A, B$ and $C$ as above, and $W^{(\infty)}(z)$ as defined in Theorem~\ref{prop:limit}. If
      $\tilde C$ is one off-diagonal block of $H_{\psi, 0}$ then
   \[
	\sigma_{1+l}(\wt C)\leq \gamma \cdot Z_l(E,F), \qquad 
	\gamma := 2 \cdot \norm{\wt C}_2. 
   \]
    \end{prop}
    
    \begin{proof}
      Due to the symmetry
      properties of the coefficients $A, B$ and $C$, 
      and to the palindromicity of $\phi(z)$ and 
      $D(z)$, we have
      \[
        G_D = \hat G_D, \qquad 
        R = \hat R, \qquad 
        a = \hat a, \qquad 
        b = \hat b.
      \]
      In this way, we 
      find that the matrix
      $
        \begin{bmatrix}
          I & I 
        \end{bmatrix} W^{(\infty)} \begin{bmatrix}
         I \\
         I
        \end{bmatrix}
      $
      satisfies a displacement relation of rank $1$
      with the same matrices of Theorem~\ref{prop:limit}. Therefore, 
      the bound on the singular values holds with
      $l$ instead of $2l$. Moreover, since the 
      matrices $G, \hat G, R$ and $\hat R$ can be
      diagonalized by means of orthogonal matrices, 
      the maximum of their spectral conditioning is
      $1$. 
    \end{proof}
    
    In order to verify the bound for this example we have carried
    out CR until convergence on a 
    $200 \times 200$ example and we have plotted the singular values of an off-diagonal block of the computed $H_0$. Then, we have estimated the bound coming from Proposition~\ref{prop:poisson} using a rational function of this form:
    \[
    r_l(z):=(z-\delta)\prod_{j=1}^{l-1}\frac{z-q_j}{z-p_j}.
    \]
    The points $q_j$ and $p_j$ are chosen with a greedy approach as the maximizer and minimizer of $r_{l-1}(z)$ in the sets $E$ and $F$ respectively. The point $\delta$ is the rightmost eigenvalue of $\varphi(z)$ inside the unit disc. The bound is compared with the one coming from Theorem~\ref{thm:zolotarev} and with the one from \cite{netna}.  
    The results are reported in Figure~\ref{fig:decaypoisson2}. 
    	
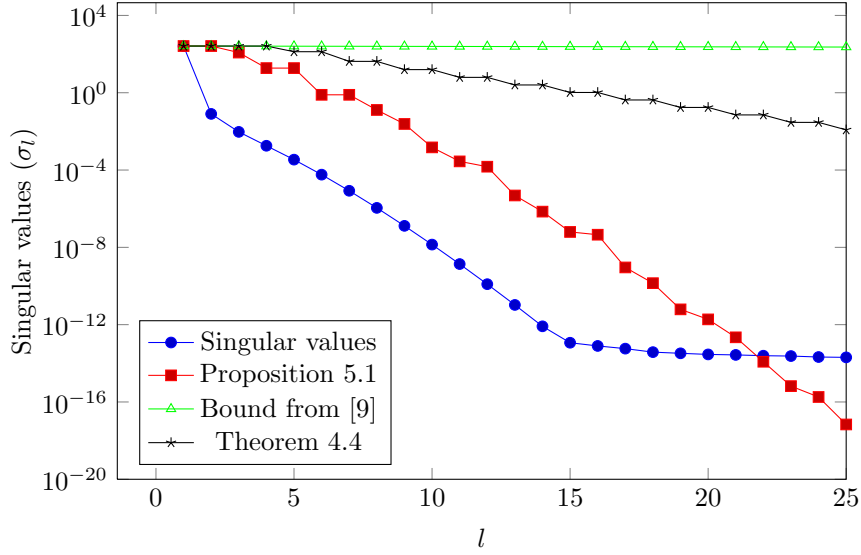
\begin{figure} 
    	\begin{center}
    	\begin{tikzpicture}
    	\begin{semilogyaxis}[
    	  xlabel=$l$,
    	  width=.76\linewidth,
    	  height=.38\textheight,
    	  ylabel=Singular values ($\sigma_l$),
    	  ymin=1e-20,
    	  xmax=25, legend pos=south west]
    	\addplot table[x index=0, y index = 1] {poisson.dat};
    	\addplot table[x index=0, y index = 2] {poisson.dat};
    	\addplot[domain=1:25,green,mark=triangle] 
    	  {260.65 * (0.995)^(x-1) };
    	  \addplot table[x index=0, y index = 3] {poisson.dat};
    	\legend{Singular values, Proposition~\ref{prop:poisson},
    		Bound from \cite{netna}, Theorem~\ref{thm:zolotarev}}
    	\end{semilogyaxis}
    	\end{tikzpicture}
    	\end{center}
    	\caption{Decay of the singular values in one of 
    		the off-diagonal blocks of  $H_0$ in 
    		the Laurent expansion of $\psi(z)$, computed
    		by means of the CR for the Poisson matrix. We have reported the actual decay and the bounds obtained by means of Proposition~\ref{prop:poisson}, the results
    		in \cite{netna}, and Theorem~\ref{thm:zolotarev}.}
    	\label{fig:decaypoisson2}

     \end{figure}
     In this case the bound from \cite{netna} is useless since the approach used there relies on a wide splitting of the eigenvalues of $\varphi(z)$. It is also interesting to note that even if the bound of Theorem~\ref{thm:zolotarev} is optimal for real intervals an ad-hoc choice for the approximant in a discrete set can deliver better results. 
  \section{Some applications}\label{sec:appl}
In this section we show some applications of CR in the case of quasiseparable blocks which for notational simplicity we refer to as Quasiseparable Cyclic Reduction (QCR for short), and we present some numerical results.

A first application concerns solving a block tridiagonal linear system of the kind $\mathcal A_nx=b$ where $\mathcal A_n=\hbox{trid}_n(B,A,C)$, the blocks $A,B,C$ are $m\times m$ matrices such that CR can be carried out with no breakdown, the right-hand side vector $b$ and the unknown vector $x$ are partitioned into $n$ blocks $b_i$ and $x_i$, respectively of size $m$.
For simplicity, assume $n=2^k-1$ so that the description of CR is simpler, for more details in the general case we refer the reader to \cite{CR}.  
  
The system can be written in the form
  \begin{equation}\label{trid_linsystem}
  \begin{bmatrix}
  A&C\\
  B&A&C\\
  &\ddots &\ddots &\ddots \\
  & &\ddots &\ddots & C\\
  & && B & A
  \end{bmatrix}
  \begin{bmatrix}
   x_1\\
   x_2\\
   \vdots\\
   \vdots\\
   x_n
   \end{bmatrix}
   =
   \begin{bmatrix}
    b_1\\
    b_2\\
    \vdots\\
    \vdots\\
    b_n
    \end{bmatrix},\qquad
    A,B,C\in\mathbb R^{m\times m},\quad x_i,b_i\in\mathbb R^m.
  \end{equation}
  An odd-even permutation of block rows and columns yields
  \[
  \left[\begin{array}{cccc|ccc}
  A & & & & C&&\\
   & A& & & B&\ddots&\\
  & &\ddots& & & \ddots& C\\
  && & A& & \ddots& B\\
  \hline
  B &C & & & A&&\\
   & \ddots&\ddots & &&\ddots&\\
  & & B&C & & & A\\
  \end{array}\right]
  \begin{bmatrix}
   x_1\\
   x_3\\
   x_5\\
   \vdots\\
   x_n\\
   \hline
   x_2\\
   \vdots\\
   x_ {n-1}
   \end{bmatrix}
   =
   \begin{bmatrix}
    b_1\\
    b_3\\
    b_5\\
    \vdots\\
    b_n\\
    \hline
    b_2\\
    \vdots\\
    b_ {n-1}
    \end{bmatrix}
  \]
  Then one step of block Gaussian elimination performed to vanish the south-western block, yields
  \[
  \left[\begin{array}{ccccc|cccc}
  A & & && & C&&\\
   & A& && & B&\ddots&\\
   & & \ddots& & &&\ddots&\ddots\\ 
  & &&\ddots& & & &\ddots &C\\
  && & &A & &&& B\\
  \hline
  & & & & & A^{(1)}& C^{(1)}&\\
  & & & & & B^{(1)}&\ddots&\ddots   \\
  & &&&&&\ddots&\ddots&C^{(1)}\\
  && & && & & B^{(1)}& A^{(1)}
  \end{array}\right]
  \begin{bmatrix}
   x_1\\
   \vdots\\
   \vdots\\
   \vdots\\
   x_n\\
   \hline
   x_2\\
   x_4\\
   \vdots\\
   \vdots\\
   x_ {n-1}
   \end{bmatrix}
   =
   \begin{bmatrix}
    b_1\\
    \vdots\\
    \vdots\\
    \vdots\\
    b_n\\
    \hline
    b_1^{(1)}\\
    b_2^{(1)}\\
    \vdots\\
    \vdots\\
    b_ {\frac{n-1}2}^{(1)}
    \end{bmatrix}
  \]
  with 
  \begin{equation}\label{cr1}
  \begin{split}
  A^{(1)}&=A-BA^{-1}C-CA^{-1}B,\\
  B^{(1)}&=-BA^{-1}B,\quad
  C^{(1)}=-CA^{-1}C,\\
  b_i^{(1)}&=b_{2i}-BA^{-1}b_{2i-1}-CA^{-1}b_{2i+1},\quad i=1,\ldots\frac{n-1}2.
  \end{split}
  \end{equation}
  
  The south-eastern block yields the system of the kind $\mathcal A_{\frac{n-1}2}x_{even}=b^{(1)}$ with $\mathcal A_{\frac{n-1}2}=\hbox{trid}_{\frac{n-1}2}(B^{(1)},A^{(1)},C^{(1)})$, 
  where $x_{even}$  denotes the subvector of $x$ formed with the even  block components, whose solution can be obtained by cyclically applying CR. Once the even block components of the block vector $x$ have been computed, they can be substituted in the first part of the linear equations so that the odd block components of $x$ are recovered. The hierarchical quasiseparability of the block matrices  makes each operation of low cost.

Thus, the first (as well as the generic) step of CR performs the following steps
  \begin{enumerate}[(i)]
  \item Given the $m\times m$ matrices $A,B,C$ compute the matrices $A^{(1)},B^{(1)},C^{(1)}$.
  \item  Given  the $m$-vectors $b_i$, $i=1,\ldots,n$, compute $b_i^{(1)}$, $i=1,\ldots,\frac{n-1}2$ by means of \eqref{cr1}.
  \item Recursively solve the system $\hbox{trid}_{\frac{n-1}2}x_{even}=b^{(1)}$ by means of CR. 
  \item Compute the odd components of the solution by mean of back substitution:
  \begin{align*}
  x_1&=A_1^{-1}(b_1-C_1x_2),\\
  x_i&=A^{-1}(b_i-Bx_{i-1}-Cx_{i+1}),\qquad i=3,5,\dots, n-2,\\
  x_n&=A^{-1}(b_1-Bx_{n-1}),
  \end{align*}
  \end{enumerate}
    
In the case where the blocks $A,B,C$ are quasiseparable, say, they are tridiagonal, not necessarily Toeplitz as in \cite{chandra}, and relying on the $\mathcal H$-matrix representation as in \cite{netna},  
in view of the preservation of the hierarchical structure of the blocks $A^{(h)},B^{(h)},C^{(h)}$ shown in the previous sections,
the cost of step (i) is  
  $O(m\log^2 m)$, while the costs of steps (ii) and (iv) is $O(nm\log m)$.  
  Therefore, indicating with $T(m,n)$ the asymptotic computational complexity of the whole algorithm with $n=2^k-1$, we have
  \[
  T(m,n)= T\left(m,\frac {n-1}2\right)+O(m\log^2 m)+O(nm\log m).
  \]
  Since $T(m,1)=O(m\log^2 m)$, we obtain $T(m,n)=O(mn \log m)+O(m\log^2m\log n)$.
  For $m=n$ this yields $T(n,n)=O(n^2\log n) +O(n\log^3 n)$.
  
  It is interesting to recall that if $\mathcal A_n$ is the discrete Laplacian where $A=\hbox{trid}_n(-1,4,-1)$, $B=C=-I$ then CR has a cost of $O(n^2\log n)$ ops \cite{fps1} while the fast Poisson solvers based on the combination of Fourier analysis and CR \cite{fps} have a cost of $O(n^2\log\log n)$ ops.
  Our approach has a slightly higher cost but covers a wider range of cases
including block tridiagonal block Toeplitz matrices with banded (not necessarily Toeplitz) blocks.

Observe that CR preserves slightly more general structures than the block tridiagonal block Toeplitz. In particular it is possible to handle the case where the first and last blocks in the main diagonal differ from the other blocks on the same diagonal \cite{CR}.
  
  \subsection{Solving certain generalized Sylvester equations}\label{sec:sylv}
  For an $m\times n$ matrix $X$ denote $x=\hbox{vec}(X)$ the $mn$-vector obtained by stacking the columns of $X$. Then, for any pair of matrices $A,B$ of compatible sizes, one has $\hbox{vec}(AB)=(I\otimes A)\hbox{vec}(B)=(B^t\otimes I)\hbox{vec}(A)$.
  
  Consider the linear matrix equation
   \begin{equation}\label{sylv}
   \sum_{i=1}^sA_i X B_i=C,\qquad A_i\in\mathbb R^{m\times m},~B_i\in\mathbb R^{n\times n},~ X,C\in\mathbb R^{m\times n},
   \end{equation}
   and suppose that  $B_i$, $i=1,\dots,s$ are tridiagonal Toeplitz matrices. 
   
   Applying the vec operator on both sides of \eqref{sylv} we get the $mn\times mn$ linear system
   \begin{equation}\label{linsys}
  Wx=c,\qquad W=\sum_{i=1}^s B_i^t\otimes A_i,\quad x=\hbox{vec} (X),\quad c=\hbox{vec} (C). 
    \end{equation}
    Since each term $B_i^t\otimes A_i$ is block tridiagonal and block Toeplitz, then the coefficient matrix of \eqref{linsys} is block tridiagonal, block Toeplitz as well.
      If the matrices $A_i$ are $k_i$-quasiseparable then the blocks of $W$ are $k$-quasiseparable with $k=\sum_{i=1}^sk_i$. If $k$ is negligible with respect to $m$ then we may solve the generalized Sylvester equation by means of QCR.
      \subsection{Numerical results}\label{sec:numexp}
A possible application of this algorithm is solving discretized partial differential equations coming from convection diffusion problems of the form
\begin{equation}\label{convdiff}
-\epsilon\Delta u(x,y)+\mathbf w\cdot \nabla u(x,y)=f(x,y),\qquad \Omega\subset \mathbb R^2
\end{equation}
where $u(x,y)$ is the unknown function, and
 we assume that the convection vector $\mathbf w$ depends only on one of the two coordinates. For simplicity we assume that it only depends on $x$. According to  \cite{simpal} we can discretize the above problem obtaining the following Sylvester equation in the matrix unknown $U$:
    \[
    \epsilon T_1 U + \epsilon U T_2+ \Phi_1 B_1 U + \Phi_2 U B_2=F.
    \]
    The independence on $y$ of the convection vector ensures that all the right factors in the previous equation are almost Toeplitz.
    The matrices $\Phi_i$ are diagonal while
    $T_i$ and $B_i$ arise from the discretization of the differential operators and they are all tridiagonal and Toeplitz with the exceptions of the first and last rows (due to the boundary conditions). The matrix $F$ contains the evaluations of the function $f$ on the discretized grid. We refer to \cite{simpal} for an in depth analysis. 
    
    We performed some numerical tests on one of the example in \cite{simpal} namely \eqref{convdiff} with $\epsilon=0.0333$ and $\mathbf w = (1+\frac{(x+1)^2}{4},0)$. The right-hand side $F$ is chosen at random. Since in this case $\Phi_2=0$ the problem is reduced to solving the Sylvester equation
    \[
    (\epsilon T_1+\Phi_1 B_1)U+U\epsilon T_2=F.
    \]
    In Figure~\ref{fig:results} and Table~\ref{tab:qcrlyap} we compare the timings and the residue of QCR with those of the function \texttt{lyap} from the control toolbox of MATLAB R2013a.
    Note that our approach can be applied even if the second coordinate of $\mathbf w$ is non zero and dependent only on $x$. In fact, in this way we retrieve a generalized Sylvester equation that can be solved with this algorithm.
    
    \begin{figure}
        	\begin{center}
        	\begin{tikzpicture}
        	\begin{loglogaxis}[
        	  xlabel=Size,
        	  width=.75\linewidth,
        	  height=.3\textheight,
        	  ylabel=Time (s),
        	  legend pos=north west]
        	\addplot table[x index=0, y index = 1] {palittasimoncini.dat};
        	\addplot table[x index=0, y index = 3] {palittasimoncini.dat};
        	\legend{QCR, {\tt lyap}};
        	\end{loglogaxis}
        	\end{tikzpicture}
        	\end{center}
        	\caption{Timings of the quasiseparable cyclic reduction (QCR) and the Sylvester solver implemented in the {\tt lyap} function in MATLAB. }\label{fig:results}
         \end{figure}
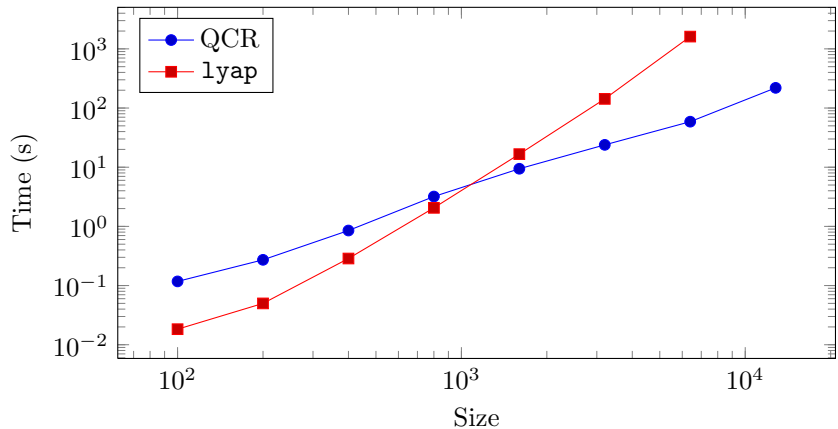
         
         \begin{table}
         \begin{center}
        	 \begin{filecontents*}{palittasimoncini.dat}
100	0.11757	2.1613e-13	0.018313	1.179e-12
200	0.27212	1.5362e-12	0.049928	5.5628e-12
400	0.85136	5.5318e-12	0.28643	5.1711e-11
800	3.1956	4.191e-11	2.0581	9.041e-11
1600	9.4175	1.2536e-10	16.632	5.6428e-10
3200	23.858	6.7801e-10	142.78	2.064e-09
6400	58.793	2.4093e-09	1612	2.9828e-08
12800	219.27	7.8008e-09	nan	nan
\end{filecontents*}
\pgfplotstabletypeset[
        	 columns/0/.style={
        	 	column name = Size},
        	 columns/1/.style={
        	 	column name = $T_{\text{QCR}}$ (s)},
        	 columns/2/.style={
        	 	column name = $Res_{\text{QCR}}$},
        	 columns/3/.style={
        	 	column name = $T_{\text{lyap}}$ (s)},
        	 columns/4/.style={
        	 	column name = $Res_{\text{lyap}}$},
        	 empty cells with ={---}  
        	 ]{palittasimoncini.dat}
         \end{center}
         \caption{Timings and residues of the Sylvester
         	equation solved by means of the quasiseparable cyclic reduction (QCR) and the Sylvester solver implemented in the {\tt lyap} function of MATLAB. The residues are computed by 
         	evaluating $\norm{\epsilon T_1 U + \epsilon U T_2 + \Phi_1 B_1 U - D}_2$. }
         	\label{tab:qcrlyap}
         \end{table}
  \section{Concluding remarks}\label{sec:conclusion}
  
  In this work we have provided an alternative analysis, 
  with respect to \cite{netna}, 
  of the numerical preservation
  of the quasiseparable structures of the matrices generated by the cyclic reduction. The theoretical results that we have
  obtained better describe the phenomenon in many
  instances coming from the applications. Examples 
  related to the solution of Sylvester
  equations arising in the discretization of
  elliptic PDEs, and from queuing theory, have
  been shown. 
  
  The
  connection between the 
  numerical preservation of the structure and the
  existence of accurate solutions of certain
  discrete rational approximation problems have been pointed out. 
  
  In the second part, the use of  CR, 
  together with hierarchical representations, 
  as a direct method for the solution of block
  tridiagonal ``almost'' Toeplitz systems has been explored and has lead to the algorithm QCR.
  This procedure has an asymptotic complexity of 
  $O(mn \log m) + O(m\log^2 m \log n)$, 
  where $n$ is the number of the blocks
  and $m$ their size. 
  Applications to the solution of elliptic differential
  equations have been shown and the effectiveness of the
  approach reported.
  
  Applications to solving certain generalized Sylvester equations, 
  of the form \[
    \sum_{i = 1}^k A_i X B_i = D, 
  \]
  have been shown in the case
  where all the blocks $B_i$s are tridiagonal Toeplitz (possibly with only the first and last row
  with different entries), and the $A_i$s have a low quasiseparable rank. Under these hypothesis, and the
  assumption that the sum of the quasiseparable ranks of
  the $A_i$s is negligible compared to $m$, 
  the complexity of the method is $O(m^2 \log m)$. 
  
 \bibliographystyle{abbrv}
 \bibliography{bibliography} 
\end{document}